\documentclass[12pt,a4paper]{amsart}

\usepackage{amsmath}
\usepackage{amssymb}
\usepackage{amsthm}
\usepackage{amscd}
\usepackage[margin=1in]{geometry}
\usepackage{fancyhdr}
\usepackage{todonotes}
\usepackage{setspace}
\usepackage[all,arc]{xy}
\usepackage{enumerate}
\usepackage{mathrsfs}
\usepackage{color}
\usepackage{hyperref}
\usepackage[hyperref,url=false,doi,style=alphabetic,maxbibnames=99]{biblatex}

\usepackage{xcolor}

\newtheorem{thm}{Theorem}[section]
\newtheorem{cor}[thm]{Corollary}
\newtheorem{prop}[thm]{Proposition}
\newtheorem{lem}[thm]{Lemma}

\newtheorem{quest}[thm]{Question}
\newtheorem*{thm*}{Theorem}
\newtheorem*{prop*}{Proposition}
\newtheorem*{lem*}{Lemma}
\newtheorem*{cor*}{Corollary}

\newtheorem{lem'}{Lemma}
\newtheorem{thm'}{Theorem}
\newtheorem{prop'}{Proposition}

\theoremstyle{definition}

\newtheorem*{exmp*}{Example}

\newtheorem*{exer*}{Exercise}

\theoremstyle{remark}
\newtheorem{rem}[thm]{Remark}

\newtheorem*{claim*}{Claim}

\newtheorem*{rem*}{Remark}
\newtheorem{rem'}{Remark}

\newtheorem*{qtn'}{Question}
\newtheorem*{soln'}{Solution}

\DeclareMathOperator{\Isom}{Isom}

\DeclareMathOperator{\dist}{dist}

\DeclareMathOperator{\Div}{Div}

\DeclareMathOperator{\Hom}{Hom}

\DeclareMathOperator{\Spec}{Spec}

\DeclareMathOperator{\SL}{SL}

\DeclareMathOperator{\C}{\mathbb{C}}

\DeclareMathOperator{\Pic}{Pic}
\DeclareMathOperator{\Cl}{Cl}

\DeclareMathOperator{\Bass}{Bass}

\DeclareMathOperator{\spin}{spin}
\DeclareMathOperator{\PSL}{PSL}

\DeclareMathOperator{\SO}{SO}
\DeclareMathOperator{\Spin}{Spin}
\DeclareMathOperator{\R}{\mathbb{R}}
\DeclareMathOperator{\Z}{\mathbb{Z}}
\renewcommand{\P}{\mathbb{P}}
\DeclareMathOperator{\D}{\mathcal{D}}
\DeclareMathOperator{\Aut}{Aut}

\renewcommand{\O}{\textnormal{O}}

\numberwithin{equation}{section}

\title{A Spectral Gap for Spinors on Hyperbolic Surfaces}

\author{Anshul Adve}
\email{aadve@princeton.edu}
\author{Vikram Giri}
\email{vikramaditya.giri@math.ethz.ch}

\begin{document}
	
		\begin{abstract}
				The purpose of this note is to construct a sequence of spin hyperbolic surfaces $\Sigma_n$ with genus going to infinity and with a uniform spectral gap for the Dirac operator. Our construction is completely explicit. In particular, the $\Sigma_n$ can be taken to be a tower of covers, with each $\Sigma_n$ an arithmetic hyperbolic surface.
		\end{abstract}
	
	\maketitle
 
\section{Introduction}

Let $\Sigma$ be a hyperbolic surface of genus $g\geq 2$ with a choice of spin structure. \emph{Throughout, all surfaces are connected, closed, and oriented.} In this paper, we are interested in obtaining uniform spectral gaps for the spin Laplacian, or equivalently the Dirac operator, on covers of $\Sigma$. The Dirac operator $\mathcal{D}$ acting on sections of the spinor bundle $S$, defined in Section~\ref{sec:Dirac}, is an elliptic, skew-adjoint, first order differential operator on $L^2(\Sigma,S)$. We define the \emph{spin Laplacian} on $\Sigma$ to be the square $\Delta^{\spin} = \mathcal{D}^2$. Since $\mathcal{D}$ is elliptic and skew-adjoint, $-\Delta^{\spin}$ is elliptic and positive-definite on $L^2(\Sigma,S)$. 

Let $\lambda_0^{\spin}(\Sigma)$ denote the smallest eigenvalue of $-\Delta^{\spin}$. It is possible, though not always the case, that $\lambda_0^{\spin} = 0$. We say that the \emph{spectral gap} for the Dirac operator on $\Sigma$ is $\sqrt{\lambda_0^{\spin}}$.
Given a cover $\Sigma'$ of $\Sigma$, the Riemannian metric and spin structure on $\Sigma$ pull back to $\Sigma'$, and using this structure we define $\lambda_0^{\spin}(\Sigma')$.
It makes sense, then, to ask for an infinite family of covers of $\Sigma$ with spectral gap for the Dirac operator uniformly bounded below.

{
Writing $\Sigma = \Gamma \backslash \mathbb{H}^2$ for some lattice $\Gamma < \PSL_2(\R)$, the spin structure determines a lift $\Gamma \hookrightarrow \SL_2(\R)$ (and such a lift determines the spin structure). Then $\lambda_0^{\spin}(\Sigma) = 0$ if and only if there exists a nonzero modular form of weight 1 for $\Gamma$.
It is notoriously difficult to tell whether such forms exist for a given group $\Gamma < \SL_2(\R)$.
In particular, the Riemann--Roch theorem gives a simple formula for the dimension of the space of modular forms of weight $k$ for $k \geq 2$, but says nothing for $k=1$.
The spectral gap for the Dirac operator can be viewed as a quantification of the non-existence of weight 1 forms.
}

We have that $0$ is in the $L^2$-spectrum of $\Delta^{\spin}$ on the upper half-plane $\mathbb{H}^2$; see Remark~\ref{rem.ltwohtwo}. So, it is easy to see that for any sequence of pointed hyperbolic surfaces $(\Sigma_n,p_n)$ where the injectivity radius at $p_n$ goes to infinity, it must be the case that $\lambda_0^\text{spin}(\Sigma_n) \to 0$ regardless of the choice of spin structure on each $\Sigma_n$. This gives an obstruction to producing {sequences with a spectral gap}, which rules out both number-theoretic constructions based on congruence surfaces, and probabilistic constructions based on random covers or random points in moduli spaces. Our main theorem is the following:

\begin{thm}[Main Theorem] \label{thm:main}

        There exist $c > 0$, a spin arithmetic hyperbolic surface $\Sigma$, and a tower of covers 
  $$\cdots \to \Sigma_2 \to \Sigma_1 \to \Sigma_0 = \Sigma$$ 
  with genus going to infinity, such that $\lambda_0^{\spin}(\Sigma_n) \geq c$ for all $n$.
	\end{thm}

The study of spinors on surfaces has a long history going back to Riemann. Here we just mention a few works relevant to us. Atiyah, in~\cite{atiyah}, connected the theory of spinors on Riemann surfaces to some classical algebraic geometry and reproved some theorems using index theory. Sarnak, in~\cite{sarnak}, related the \emph{determinant} of the spin Laplacian to the central value of a twisted Selberg zeta function.
In the recent work~\cite{palspin}, the authors use methods from conformal field theory to prove interesting, rigorous bounds relating the spectral gap for spinors to that of the Laplacian on functions. 

% \subsection{Applications via the trace formula}
\subsection{Relation to coclosed 1-form spectra on hyperbolic 3-manifolds}\label{s.3f}
In the work \cite{aaglz}, the authors construct families of hyperbolic $3$-manifolds with volume going to infinity with a uniform spectral gap on coclosed $1$-forms. By Hodge theory, these are all rational homology spheres. The setting of that paper and the present one share a lot of common features. This stems from the representation theoretic fact that the relevant representations of $\Isom^+(\mathbb{H}^2)$ and $\Isom^+(\mathbb{H}^3)$ having the smallest Casimir eigenvalue are tempered in both cases. In both works, we are able to construct manifolds with volume going to infinity having a uniform spectral gap above this smallest tempered eigenvalue. 

A uniform gap above the smallest tempered eigenvalue is not possible for the Laplacian on functions by the work of Huber~\cite{Huber}. Indeed, there one can use the fact that the terms in the geometric side of the Selberg trace formula involve the lengths of closed geodesics and are, consequently, positive. However, in the case of coclosed $1$-forms on hyperbolic $3$-manifolds or spinors on hyperbolic surfaces, the analogous Selberg trace formulae involve the {complex} length or the twist by the character arising from the spin structure, respectively, and thus the terms on the geometric side do not have a fixed sign. 

\subsection{The bass note spectrum}
The \emph{bass note spectrum} for the spin Laplacian on a family $\mathcal{F}$ of spin hyperbolic surfaces is the closed subset of $[0,\infty)$ defined as
$$ \Bass^{\text{spin}}_{\mathcal{F}} = \overline{\{\lambda_0^\text{spin}(\Sigma) : \Sigma\in\mathcal{F}\}}\,. $$
The bass note spectrum can be analogously defined for other differential operators on families of geometric structures: the reader is encouraged to look at the lectures of Sarnak~\cite{sarnak_chern} for a nice overview and further references.

When $\mathcal{F} = \text{hyp}$ consists of all spin hyperbolic surfaces (closed, connected, and oriented as usual), there exists a constant $C>0$ such that $\Bass_{\text{hyp}}^{\text{spin}} \subseteq [0,C]$. The existence of $C$ means that there is a universal upper bound for the first eigenvalue of the spin Laplacian on any spin hyperbolic surface. This bound relies on the fact that one can isometrically embed a hyperbolic disc of a fixed small radius into any hyperbolic surface $\Sigma$; see Appendix \ref{app} for a simple proof. The bound then follows from the variational characterization of $\lambda_0^{\spin}$, using a test function supported on the embedded disc. We remark that this method yields upper bounds on the spectral gap for a large class of invariant differential operators acting on sections of various bundles on $\Sigma$.

When $\mathcal{F}=\text{arith}$ consists of spin arithmetic hyperbolic surfaces, our main Theorem~\ref{thm:main} implies, in view of the above upper bound, that $\Bass^{\text{spin}}_{\text{arith}}$ has a non-zero limit point. 

\subsection{Relation to number fields}
{This subsection is largely speculatory in nature. Consider the twisted Selberg zeta function for a spin hyperbolic surface as considered by Sarnak in~\cite{sarnak}. Our main theorem~\ref{thm:main} can be reformulated as follows: there exists a tower of covers of a spin arithmetic hyperbolic surface such that the zeros of their corresponding twisted zeta functions are uniformly bounded away from the central point $1/2$. Note that for the usual Selberg zeta functions, the trace formula implies that some zeros must approach $1/2$ along any sequence of surfaces with volume going to infinity.
This is equivalent to the result of Huber \cite{Huber} mentioned above.

This perspective allows us to contemplate the existence of analogous phenomena in the realm of number fields via the knots-and-primes correspondence. This is an analogy between $3$-manifolds and global fields. For our purposes, consider the following analogies between a number field $K$ and (the unit tangent bundle of) a hyperbolic surface $M$:
\begin{center}
\begin{tabular}{|c|c|}\hline
    $\log($discriminant of $K)$ & The genus of $M$\\\hline
    {$\Spec \mathcal{O}_K$}
    %The ring of integers $\mathcal{O}_K$
    & The unit tangent bundle $T^1M$\\\hline
    The Guinand--Weil formula & The Selberg trace formula\\\hline
    The Dedekind zeta function $\zeta_K(s)$ & The Selberg zeta function $Z_M(s)$\\\hline
\end{tabular}
\end{center}
A choice of spin structure is a choice of a degree $2$ cover of $3$-manifolds ${P} \to T^1M$ satisfying a certain nontriviality condition; see Section~\ref{subsec:spinors}. If we regard this as being analogous to a degree $2$ (unramified) extension $L/K$ of number fields, then our main theorem~\ref{thm:main} suggests a positive answer to the following question.
\begin{quest}
    Do there exist infinitely many pairs of number fields $(L,K)$ such that $L/K$ is a degree $2$ (unramified) extension and the zeros of $\zeta_L(s)/\zeta_K(s)$ are uniformly bounded away from the central point $1/2$?
\end{quest}
Again, for the Dedekind zeta functions, it is known under GRH that the lowest-lying zero must approach $1/2$ along any sequence of distinct number fields; see~\cite[GRH Proposition 5.2]{tsvl}. This fact is proved using that the ``geometric side'' of a (limiting) Guinand--Weil formula has terms with a fixed sign in analogy with the discussion in Section~\ref{s.3f}.} However, a positive answer to the above question {would tell} us that such a statement is no longer true for $L$-functions of quadratic Hecke characters.

\subsection{Questions}
We end this introduction with a series of open questions motivated by other works on similar problems.
\begin{quest}
    Does there exist a sequence of spin hyperbolic surfaces with a uniform spectral gap for both functions and spinors?
\end{quest}
The above question is motivated by the theory of high dimensional expanders~\cite{lub} asking for uniform joint gaps in the spectra of more than one self-adjoint operator. An analogous question has been asked (c.f.~\cite[Question~1]{aaglz}) in the context of hyperbolic $3$-manifolds.

\begin{quest}
    Can one have the spinor spectral gap be non-zero but exponentially small in volume along a sequence of covers? 
\end{quest}
Here we are motivated by the analogy between spinors on hyperbolic surfaces and coclosed $1$-forms on hyperbolic $3$-manifolds. In the $3$-manifold setting, Rudd in~\cite{rudd} has proved that one can indeed have an exponentially small gap (though his construction is not a sequence of covers, but just a sequence with volume going to infinity).

\begin{quest}
    Describe the structure of the bass note spectrum for spinors on arithmetic surfaces. In particular, do we have $[0,c] \subset \Bass^{\textnormal{spin}}_{\textnormal{arith}}$ for some $c>0$?
\end{quest}
In the work~\cite{magee}, Magee shows that the maximal interval $[0,1/4]$ is contained in the bass note spectrum for functions on arithmetic hyperbolic surfaces.

\begin{subsection}*{Acknowledgements}
We wish to thank Peter Sarnak for suggesting this problem, and for telling us about the Riemann vanishing theorem and Bloch wave theory, which inspired our proof. Both authors thank the Simons Collaboration on the Localization of Waves for their support during the completion of this project.
The first author also thanks the National Science Foundation for its support via the Graduate Research Fellowship Program under Grant No. DGE-2039656.
\end{subsection}

\section{Setup and Preliminaries}

\subsection{Spinors on Riemannian surfaces} \label{subsec:spinors}

Let $\Sigma$ be a Riemannian surface. A spin structure on $\Sigma$ is a double cover $P$ of the unit tangent bundle $T^1\Sigma$, such that the $\SO(2)$-action on $T^1\Sigma$ by rotation lifts to a $\Spin(2)$-action on $P$ making $P$ into a principal $\Spin(2)$-bundle. Let $\pi \colon \Spin(2) \to \SO(2)$ be the canonical projection map, and $\rho \colon \Spin(2) \to \SO(2)$ the pointwise square root of $\pi$. If one identifies $\Spin(2) \simeq U(1) \simeq \SO(2)$, then $\pi,\rho \colon U(1) \to U(1)$ are given by $\pi(z) = z^2$ and $\rho(z) = z$. Thus $\rho$ is an isomorphism of groups. Viewing $\rho$ as a representation of $\Spin(2)$ on $\R^2$, one can form the associated bundle $S = P \times_{\rho} \R^2$. 
By defintion, this is the quotient of $P \times \R^2$ by the diagonal action of $\Spin(2)$, where $\Spin(2)$ acts via $\rho$ on $\R^2$.
The bundle $S$ is the \emph{spinor bundle} on $\Sigma$. It is a real vector bundle of rank $2$, and it inherits a metric from the standard metric on $\R^2$. Sections of $S$ are called \emph{spinor fields}, or simply \emph{spinors}. 

\subsection{Theta characteristics}
	
	Let $\Sigma$ be a spin Riemannian surface with spinor bundle $S$. In order to define the Dirac operator $\D$ on sections of $S$, it suffices to define $\D$ on sections of the complexification $S_{\C}$, and then check that $\D$ commutes with complex conjugation. Recall that $S = P \times_{\rho} \R^2$ is the bundle associated with the representation $\rho$ of $\Spin(2)$. The complexification of this representation decomposes as $\rho_{\C} \simeq \rho_+ \oplus \rho_-$, where $\rho_{\pm} \colon \Spin(2) \to \C^{\times}$ are the complex one-dimensional representations $\rho_{\pm} = \sigma_{\pm} \circ \rho$, with $\sigma_{\pm} \colon \SO(2) \to \C^{\times}$ given by
	\begin{align*}
		\sigma_{\pm}
		\begin{pmatrix}
			\cos\theta & -\sin\theta \\
			\sin\theta & \cos\theta
		\end{pmatrix}
		= e^{\pm i\theta}.
	\end{align*}
	Therefore $S_{\C} \simeq S_+ \oplus S_-$, where $S_{\pm}$ are the complex line bundles $S_{\pm} = P \times_{\rho_{\pm}} \C$. Complex conjugation on $S_{\C}$ interchanges $S_{\pm}$, taking $(p,z) \in S_+$ to $(p,\overline{z}) \in S_-$.
	
	Recalling that $\pi = \rho^2$ denotes the canonical projection from $\Spin(2)$ to $\SO(2)$, we note that $\rho_{\pm}^2 = \sigma_{\pm} \circ \pi$. So, we see that there is a natural map of complex line bundles from $S_+^{\otimes 2}$ to the canonical bundle $K_{\Sigma}$ given as the composition
	\begin{align*}
		S_+^{\otimes 2}
		\simeq P \times_{\rho_+^2} \C
		= P \times_{\sigma_+ \circ \pi} \C
		\simeq T^1\Sigma \times_{\sigma_+} \C
		\to K_{\Sigma},
	\end{align*}
	where the final map takes $(v,z) \in T^1\Sigma \times_{\sigma_+} \C$ to the cotangent vector $\xi \in K_{\Sigma}$ satisfying $\xi(v) = z$. This composition $S_+^{\otimes 2} \to K_{\Sigma}$ is an isomorphism. Equip $S_+$ with the holomorphic structure induced from this isomorphism. Then $S_+$ is a holomorphic line bundle squaring to the canonical bundle, or in other words, $S_+$ is a \textit{theta characteristic}.
	
	Conversely, suppose given a theta characteristic $K_{\Sigma}^{1/2}$. Denoting the dual of a complex line bundle $L$ by $L^*$, there is a quadratic map $q \colon (K_{\Sigma}^{1/2})^* \to T\Sigma$ given as the composition of the squaring map $(K_{\Sigma}^{1/2})^* \to K_{\Sigma}^*$ followed by the isomorphism (of real rank $2$ vector bundles) $K_{\Sigma}^* \simeq T\Sigma$. The preimage $P = q^{-1}(T^1\Sigma)$ is then a spin structure --- it is a double cover of $T^1\Sigma$ which is a principal bundle for the induced $\Spin(2)$-action.
	
	We have now described how to construct theta characteristics from spin structures and vice versa. It is not difficult to check that these constructions are inverses of each other, and hence establish a bijection between spin structures and theta characteristics \cite{atiyah}.

\subsection{The Dirac operator}\label{sec:Dirac}
	For any holomorphic line bundle $L$, the complex analytic $\overline{\partial}$ operator is a first order differential operator taking sections of $L$ to sections of $\overline{K_{\Sigma}} \otimes L$. Let $\D_-$ be the first order differential operator taking sections of $S_+$ to sections of $S_-$, by
	\begin{align*}
		\D_-
		\colon S_+ \xrightarrow{\overline{\partial}} \overline{K_{\Sigma}} \otimes S_+
		\simeq \overline{S_+^{\otimes 2}} \otimes S_+
		\simeq P \times_{\overline{\rho_+^2} \rho_+} \C
		= P \times_{\rho_-} \C
		= S_-.
	\end{align*}
	Define $\D_+$ from sections of $S_-$ to sections of $S_+$ by complex conjugation symmetry, so $\D_+s = \overline{\D_- \overline{s}}$. Finally, define the Dirac operator on sections of $S_{\C} \simeq S_+ \oplus S_-$ by
	\begin{align*}
		\D =
		\begin{pmatrix}
			0 & \D_+ \\
			\D_- & 0
		\end{pmatrix}.
	\end{align*}
	By construction, $\D$ commutes with complex conjugation, so $\D$ descends to an operator on sections of $S$. Since $\overline{\partial}$ is elliptic, $\D$ is elliptic. One can check that the adjoint of $\D_{\pm}$ is $-\D_{\mp}$. It follows that $\D$ is skew-adjoint as an operator on $L^2(\Sigma,S)$.
	
	The spin Laplacian $\Delta^{\spin} = \D^2$ preserves the decomposition $S_{\C} \simeq S_+ \oplus S_-$. From this we obtain the following trivial proposition which will be useful later.
	
	\begin{prop}
		The spectrum of $\Delta^{\spin}$ on the real Hilbert space $L^2(\Sigma,S)$ is the same as on the complex Hilbert space $L^2(\Sigma,S_+)$, but with twice the multiplicity. In particular, $\lambda_0^{\spin}(\Sigma)$ is the smallest eigenvalue of $-\Delta^{\spin}$ on $L^2(\Sigma,S_+)$.
	\end{prop}

	\begin{proof}
		By complexification, the spectrum of $\Delta^{\spin}$ on the real Hilbert space $L^2(\Sigma,S)$ is the same as on the complex Hilbert space $L^2(\Sigma,S_{\C}) \simeq L^2(\Sigma,S_+) \oplus L^2(\Sigma,S_-)$. Since $\Delta^{\spin}$ preserves this decomposition, the spectrum on $L^2(\Sigma,S_{\C})$ is the union (with multiplicity) of the spectra on $L^2(\Sigma,S_+)$ and $L^2(\Sigma,S_-)$. These latter two spectra are the same because $L^2(\Sigma,S_+)$ and $L^2(\Sigma,S_-)$ are intertwined by complex conjugation.
	\end{proof}

 \begin{rem}\label{rem.ltwohtwo}
     The spin Laplacian $\Delta^{\spin}$ acting on sections of the theta characteristic $S_+$ is not quite the Laplacian $\Delta_{(1/2)}$ on weight $1$ forms that one gets from the representation theoretic point of view on hyperbolic surfaces as in~\cite{palspin}. In fact, one has $-\Delta_{(1/2)} = -\Delta^{\spin} + 1/4 $. In~\cite{palspin}, it is explained that the temperedness of the principal series representation $\mathcal{P}^{-}_0$ {of $\SL_2(\R)$} is the reason for $0$ being in the spectrum of $-\Delta^{\spin}$ on $L^2(\mathbb{H}^2,S)$ for the (unique) spin structure on the hyperbolic plane.
 \end{rem}

\subsection{Fourier analysis on abelian covers}\label{sec:bloch}
	
	Let $\Sigma$ be a spin Riemannian surface of genus $g$ with spinor bundle $S$. Let
	\begin{align*}
		\widehat{H_1(\Sigma,\Z)} = \Hom(H_1(\Sigma,\Z),U(1))
	\end{align*}
	denote the Pontryagin dual of the first homology of $\Sigma$. This is a torus of dimension $2g$. Consider a finite subgroup $H$ of this torus. By duality, $H$ corresponds to a finite quotient $H^{\vee}$ of $H_1(\Sigma,\Z)$. Let $\Sigma_H$ be the abelian cover of $\Sigma$ corresponding to $H^{\vee}$, so $\Sigma_H$ has fiber $H^{\vee}$ with monodromy the translation action of $H_1(\Sigma,\Z)$ on $H^{\vee}$. By Fourier analysis,
	\begin{align*}
		L^2(\Sigma_H,S_+)
		= \bigoplus_{\chi \in H} L^2(\Sigma, S_+ \otimes L_{\chi}),
	\end{align*}
	where $L_{\chi}$ is the flat holomorphic line bundle on $\Sigma$ whose sections have monodromy $\chi$. The spin Laplacian preserves this decomposition, so
	\begin{align*}
		\lambda_0^{\spin}(\Sigma_H)
		= \min_{\chi \in H} \lambda_0^{\spin}(\Sigma,\chi),
	\end{align*}
	where $\lambda_0^{\spin}(\Sigma,\chi)$ is the smallest eigenvalue of $-\Delta^{\spin}$ on $L^2(\Sigma, S_+ \otimes L_{\chi})$.
	
	\begin{prop} \label{prop:bloch}
		Let $T \subseteq \widehat{H_1(\Sigma,\Z)}$ be a closed subtorus with $n$-torsion subgroup denoted $T[n]$. Then $\lambda_0^{\spin}(\Sigma_{T[n]})$ is bounded below by a positive constant independent of $n$ if and only if $S_+ \otimes L_{\chi}$ has no nonzero holomorphic section for all $\chi \in T$.
	\end{prop}

	\begin{proof}
		Let $T[\infty] \subseteq T$ denote the subset of all torsion points. By the discussion above,
		\begin{align*}
			\min_n \lambda_0^{\spin}(\Sigma_{T[n]})
			= \min_{\chi \in T[\infty]} \lambda_0^{\spin}(\Sigma,\chi).
		\end{align*}
		Lemma \ref{lem:eigenvalue_cty} below says that $\lambda_0^{\spin}(\Sigma,\chi)$ depends continuously on $\chi$, so by density of $T[\infty]$ in $T$,
		\begin{align*}
			\min_n \lambda_0^{\spin}(\Sigma_{T[n]})
			= \min_{\chi \in T} \lambda_0^{\spin}(\Sigma,\chi)
		\end{align*}
		(the minimum on the right hand side is attained because $T$ is compact). Therefore $\lambda_0^{\spin}(\Sigma_{T[n]})$ is uniformly bounded below if and only if $\lambda_0^{\spin}(\Sigma,\chi) \neq 0$ for all $\chi$. By Lemma \ref{lem:harmonic_implies_holo} below, the proposition follows.
	\end{proof}
	
	\begin{lem} \label{lem:eigenvalue_cty}
		The eigenvalue $\lambda_0^{\spin}(\Sigma,\chi)$ depends continuously on $\chi \in \widehat{H_1(\Sigma,\Z)}$.
	\end{lem}

	This sort of continuity holds in great generality. In particular, the proof below shows that all the eigenvalues, not just the first, vary continuously.

    {
    \begin{proof}
        {
        Let us check continuity at $\chi_0 \in \widehat{H_1(\Sigma,\Z)}$. By a partition of unity argument, one can choose for each $\chi \approx 1$ a smooth section $\varphi_{\chi}$ of $L_{\chi}$ such that $\varphi_1 = 1$ and $\varphi_{\chi}$ depends smoothly on $\chi$. Then for $\chi \approx 1$, the section $\varphi_{\chi}$ is nowhere vanishing. Thus for $\chi \approx \chi_0$, multiplication by $\varphi_{\chi_0^{-1}\chi}$ gives an isomorphism { (of $C^{\infty}$ complex line bundles)} from $S_+ \otimes L_{\chi_0}$ to $S_+ \otimes L_{\chi}$. Let $\Delta_{\chi}^{\spin}$ be the differential operator on sections of $S_+ \otimes L_{\chi_0}$ given by conjugating $\Delta^{\spin}$ by this isomorphism. Then $\lambda_0^{\spin}(\Sigma,\chi)$ is the first eigenvalue of $-\Delta_{\chi}^{\spin}$.
        }

        Equivalently, $(1+\lambda_0^{\spin}(\Sigma,\chi))^{-1}$ is the largest eigenvalue of the resolvent $(1-\Delta_{\chi}^{\spin})^{-1}$.
        Since $\Delta_{\chi}^{\spin}$ is elliptic, its domain as an unbounded self-adjoint operator on $L^2(\Sigma, S_+ \otimes L_{\chi_0})$ is $H^2(\Sigma, S_+ \otimes L_{\chi_0})$.
        Thus by definition, the above resolvent is the operator on $L^2$ given as the inverse of the isomorphism $1-\Delta_{\chi}^{\spin} \colon H^2 \to L^2$. Since $1-\Delta_{\chi}^{\spin}$ depends continuously on $\chi$ with respect to the operator norm topology on $\mathcal{B}(H^2,L^2)$, its inverse depends continuously with respect to the operator norm topology on $\mathcal{B}(L^2,H^2)$, and in particular on $\mathcal{B}(L^2,L^2)$. Again since $\Delta_{\chi}^{\spin}$ is elliptic, and since $\Sigma$ is a compact surface, the resolvent is a compact operator on $L^2$. It follows from Lemma~\ref{lem:cpt_eigenvalue_cty} below that the spectrum of the resolvent depends continuously on $\chi$. Consequently $\lambda_0^{\spin}(\Sigma,\chi)$ depends continuously on $\chi$.
    \end{proof}

    Given a Banach space $X$, let $\mathcal{K}(X)$ denote the space of compact operators on $X$ with the operator norm topology.
    Let $\mathcal{P}(\C)$ be the space of subsets of $\C$ with the Hausdorff metric.

    \begin{lem} \label{lem:cpt_eigenvalue_cty}
        Let $X$ be a Banach space. Then the map $\mathcal{K}(X) \to \mathcal{P}(\C)$ taking a compact operator to its spectrum is continuous.
    \end{lem}

    \begin{proof}
        For $X$ finite-dimensional, this is proven in Chapter 2, Section 5.1 of \cite{Kato} using the holomorphic functional calculus.
        The same proof works for compact operators in infinite dimension.
    \end{proof}
    }

	\begin{lem} \label{lem:harmonic_implies_holo}
		A section $s$ of $S_+ \otimes L_{\chi}$ (for any $\chi$) is holomorphic if and only if $\Delta^{\spin} s = 0$.
	\end{lem}

	\begin{proof}
		The easy direction is when $s$ is holomorphic. Then $\Delta^{\spin} s = \D_+ \D_- s = 0$, because $\D_-$ is defined as a composition of operators beginning with $\overline{\partial}$.
		
		In the reverse direction, suppose $\Delta^{\spin} s = 0$. Then by elliptic regularity, $s$ is smooth. Thus we can write
		\begin{align*}
			0 = \langle \Delta^{\spin} s,s \rangle_{L^2(\Sigma,S_+ \otimes L_{\chi})}
			= \langle \D_+ \D_- s,s \rangle_{L^2(\Sigma,S_+ \otimes L_{\chi})}
			= -\|\D_- s\|_{L^2(\Sigma,S_- \otimes L_{\chi})}^2,
		\end{align*}
		where in the last equality we have integrated by parts, using that $\D_+^* = -\D_-$. Hence $\D_-s = 0$. Since $\D_-$ is given by $\overline{\partial}$ followed by an isomorphism, we conclude that $\overline{\partial}s = 0$, i.e., $s$ is holomorphic.
	\end{proof}

\section{The explicit construction} \label{sec:construction}
	
	In this section, we construct (one example of) the objects whose existence is asserted in Theorem \ref{thm:main}. Their desired properties are proven in Sections \ref{sec:gap_pf} and \ref{sec:uniformization}.
	
	From now on, ``curve" is short for ``smooth complex projective curve." However, we will often present curves by affine models, and denote points on a curve by their coordinates in affine space in the given model.
	
	Let $\Sigma$ be the genus 2 curve $y^2 = x^6-1$. We will see in Section \ref{sec:uniformization} that this is an arithmetic Riemann surface. Let $\frac{1}{2}K_{\Sigma}$ be the divisor
	\begin{align*}
		(\zeta_6,0)  + (-\zeta_6,0) - (1,0)
	\end{align*}
	on $\Sigma$, where $\zeta_n = e^{2\pi i/n}$, and let $K_{\Sigma}^{1/2} = \mathcal{O}_{\Sigma}(\frac{1}{2}K_{\Sigma})$ be the corresponding line bundle. By Corollary \ref{cor:2(branch_pt)} in Section \ref{sec:genus_2}, we see that $2(\frac{1}{2}K_{\Sigma})$ is linearly equivalent to the canonical divisor, so $K_{\Sigma}^{1/2}$ is a theta characteristic.
	
	Let $E$ be the genus 1 curve $y^2 = x^4 - x$, and let $f \colon \Sigma \to E$ be the degree 2 map
	\begin{align*}
		f(x,y)
		= (x^2,xy).
	\end{align*}
	Choose a one-dimensional closed subtorus $T_E$ of $\widehat{H_1(E,\Z)}$ such that the unique nonzero 2-torsion point of $T_E$ is the monodromy character of the multi-valued function $\sqrt{1-x^{-1}}$ on $E$ (note that the zeros and poles on $E$ of $1-x^{-1}$ are of order 2, so locally there is always a meromorphic square root with no need for branch cuts). Let $T$ be the image of $T_E$ in $\widehat{H_1(\Sigma,\Z)}$ under the pullback map $f^*$. Then $T$ is a one-dimensional closed subtorus of $\widehat{H_1(\Sigma,\Z)}$. Let $\Sigma_n = \Sigma_{T[2^n]}$. Then the $\Sigma_n$ form a tower of covers of $\Sigma_0 = \Sigma$.
	
	\begin{thm} \label{thm:no_holo_sections}
		For all $\chi \in T$, there are no nonzero holomorphic sections of $K_{\Sigma}^{1/2} \otimes L_{\chi}$.
	\end{thm}

	By Proposition \ref{prop:bloch} and the correspondence between spin structures and theta characteristics, Theorem \ref{thm:no_holo_sections} implies Theorem \ref{thm:main}. It remains to prove Theorem \ref{thm:no_holo_sections}.
	
	\section{Standard lemmas about genus 2 curves} \label{sec:genus_2}
	
	Let $X$ be the curve $y^2 = h(x)$ for some monic polynomial $h$. For most of the results in this section, we will assume $X$ has genus 2, which is the case if and only if $h$ has degree 5 or 6.
	
	In this section, let $K_X$ denote the canonical divisor on $X$ (well defined up to linear equivalence), as opposed to the canonical line bundle on $X$. Let $p \colon X \to \mathbb{P}^1 = \C \cup \{\infty\}$ be the projection $p(x,y) = x$.
	
	\begin{lem} \label{lem:deg_2_divisor_lemma}
		Let $X$ have genus $2$. Let $D_1,D_2$ be distinct effective divisors on $X$ of degree $2$. Then $D_1 \sim D_2$ if and only if $D_1,D_2$ are both canonical, if and only if each $D_i$ is the sum of the two points (counting multiplicity) in some fiber of $p$.
	\end{lem}

	\begin{proof}
		Suppose $D_1 \sim D_2$. Then there is a non-constant function in $H^0(\mathcal{O}(D_1))$, so by Riemann--Roch,
		\begin{align*}
			2 \leq \ell(D_1)
			= 1 + \ell(K_X-D_1).
		\end{align*}
		Since $X$ has genus $2$, the canonical divisor has degree $2$, so $K_X-D_1$ has degree $0$. However, the above inequality says that $\mathcal{O}(K_X-D_1)$ has a nonzero section. This forces $D_1 \sim K_X$. Similarly, $D_2 \sim K_X$. Thus the first equivalence holds.
		
		To prove the second equivalence, it is enough (now that we know the first equivalence) to show that an effective divisor $D$ on $X$ is canonical if and only if $D = p^*(a)$ for some $a \in \mathbb{P}^1$. Note first that $p^*(a) \sim p^*(b)$ for all $a,b \in \mathbb{P}^1$, because $(a) \sim (b)$ as divisors on $\mathbb{P}^1$. Thus by the first equivalence, $p^*(a) \sim K_X$ for each $a \in \mathbb{P}^1$. It remains to show that if $D$ is canonical, then $D = p^*(a)$ for some $a \in \mathbb{P}^1$. Assuming $D$ is canonical, $D \sim K_X \sim p^*(\infty)$. This means that $D-p^*(\infty)$ is the divisor of a nonzero function in $H^0(\mathcal{O}(p^*(\infty)))$. This space of functions contains $1$ and $x$, and has dimension 2 by Riemann--Roch. Thus $D-p^*(\infty)$ is the divisor of a function of $x$. Hence $D = p^*D'$ for some divisor $D'$ on $\mathbb{P}^1$. Since $D$ is effective of degree 2, it follows that $D'$ is effective of degree 1, so $D' = (a)$ for some $a \in \mathbb{P}^1$, and $D = p^*(a)$ is of the desired form.
	\end{proof}

	\begin{rem} \label{rem:any_genus}
		Part of the statement of Lemma \ref{lem:deg_2_divisor_lemma} is that if each $D_i$ is the sum of the two points (counting multiplicity) in some fiber of $p$, then $D_1 \sim D_2$. The proof of this part does not require that $X$ has genus $2$.
	\end{rem}

	\begin{cor} \label{cor:2(branch_pt)}
		Let $w_1,w_2 \in X$ be branch points of $p$. Then $2(w_1) \sim 2(w_2)$ as divisors on $X$. In addition, if $X$ has genus 2, then $2(w_1)$ is canonical.
	\end{cor}

	\begin{proof}
		This follows from Lemma \ref{lem:deg_2_divisor_lemma}, Remark \ref{rem:any_genus}, and the fact that if $w$ is a branch point, then $2(w) = p^*(p(w))$ as divisors on $X$.
	\end{proof}

	We note, for concreteness, that $(a,0) \in X$ is a branch point of $p$ for each root $a$ of $h$. When $\deg h$ is even, there are no other branch points. When $\deg h$ is odd, $p^{-1}(\infty)$ consists of one point, and this point is the only other branch point.
	
	\section{Proof of spectral gap} \label{sec:gap_pf}
	
	As stated at the end of Section \ref{sec:construction}, in order to prove Theorem \ref{thm:main}, it suffices to prove Theorem \ref{thm:no_holo_sections}. By definition, each element of $T$ is of the form $f^*\chi$ for some $\chi \in T_E$. The line bundle $L_{f^*\chi}$ on $\Sigma$ is the pullback via $f$ of the line bundle $L_{\chi}$ on $E$. Thus Theorem \ref{thm:no_holo_sections} is equivalent to the statement that $K_{\Sigma}^{1/2} \otimes f^*L_{\chi}$ has no nonzero holomorphic sections for all $\chi \in T_E$. We will prove this assuming four lemmas, and then afterward we will go back and prove the lemmas.
	
	Given a holomorphic line bundle $L$, we write $\Div(L)$ for its divisor (which is well-defined up to linear equivalence).
	
	\begin{lem} \label{lem:deg(L_chi)=0}
		Flat line bundles on compact Riemann surfaces always have degree zero.
	\end{lem}
	
	\begin{lem} \label{lem:chi_2_divisor}
		Let $\chi_2 \in T_E$ be the unique nonzero $2$-torsion character. Then
		\begin{align} \label{eqn:Div(L_{chi_2})}
			\Div(L_{\chi_2})
			\sim (1,0) - (0,0)
		\end{align}
		on $E$, and
		\begin{align} \label{eqn:Div(f^*L_{chi_2})}
			\Div(f^*L_{\chi_2})
			\sim (1,0) - (-1,0)
		\end{align}
		on $\Sigma$.
	\end{lem}

	\begin{lem} \label{lem:injective}
		The map $T_E \to \Pic(\Sigma)$ by $\chi \mapsto f^*L_{\chi}$ is injective.
	\end{lem}

	Let $E'$ be the elliptic curve $y^2 = x^3 - 1$, and let $f' \colon \Sigma \to E'$ be the map
	\begin{align*}
		f'(x,y) = (x^2,y).
	\end{align*}
	
	\begin{lem} \label{lem:zero_map}
		The map $f_*' f^* \colon \Cl^0(E) \to \Cl^0(E')$ is the zero map.
	\end{lem}

	\begin{proof}[Proof of Theorem \ref{thm:no_holo_sections} assuming Lemmas \ref{lem:deg(L_chi)=0} through \ref{lem:zero_map}]
		Suppose for a contradiction that $K_{\Sigma}^{1/2} \otimes f^*L_{\chi}$ has a nonzero holomorphic section for some $\chi \in T_E$. This means that the divisor
		\begin{align*}
			\Div(K_{\Sigma}^{1/2} \otimes f^*L_{\chi})
			\sim \frac{1}{2}K_{\Sigma} + f^*\Div(L_{\chi})
		\end{align*}
		is effective. By Lemma \ref{lem:deg(L_chi)=0}, this divisor has degree 1, so effectivity means that it is linearly equivalent to $(q)$ for some $q \in \Sigma$. In symbols,
		\begin{align} \label{eqn:effective}
			\frac{1}{2}K_{\Sigma} + f^*\Div(L_{\chi})
			\sim (q).
		\end{align}
		Pushing this equivalence forward along $f'$ and using Lemma \ref{lem:zero_map} gives
		\begin{align*}
			(f'(q))
			\sim f'_*\Big(\frac{1}{2}K_{\Sigma}\Big)
			= 2(\zeta_3,0) - (1,0)
			\sim (1,0),
		\end{align*}
		where the final $\sim$ is by Corollary \ref{cor:2(branch_pt)}.
		Since $E'$ has genus 1, we have $\Cl^1(E') = E'$, so $f'(q) = (1,0)$, and hence $q = (\pm 1,0)$ for some choice of sign.
		Plugging this into \eqref{eqn:effective},
		\begin{align} \label{eqn:f^*Div(L_chi)}
			f^* \Div(L_{\chi})
			\sim (q) - \frac{1}{2}K_{\Sigma}
			= (\pm 1,0) - \frac{1}{2}K_{\Sigma}
			= (\pm 1,0) + (1,0) - (\zeta_6,0) - (-\zeta_6,0).
		\end{align}
		By Corollary \ref{cor:2(branch_pt)}, the right hand side is 2-torsion, so $f^*L_{\chi}$ is 2-torsion in $\Pic(\Sigma)$.
		It follows from Lemma~\ref{lem:injective} that $\chi$ is 2-torsion, so either $\chi = 1$ or $\chi = \chi_2$ is the unique nonzero 2-torsion element in $T_E$.
		Thus by Lemma~\ref{lem:chi_2_divisor},
		\begin{align*}
			f^*\Div(L_{\chi})
			\sim
			\begin{cases}
				(1,0) - (-1,0) &\text{if } \chi = \chi_2, \\
				0 &\text{if } \chi = 1.
			\end{cases}
		\end{align*}
		Plugging this into \eqref{eqn:f^*Div(L_chi)}, we get
		\begin{align*}
			(\pm 1,0) + (\pm 1,0) \sim (\zeta_6,0) + (-\zeta_6,0)
		\end{align*}
		for some choice of signs. This contradicts Lemma \ref{lem:deg_2_divisor_lemma}, because the two points $(\zeta_6,0)$ and $(-\zeta_6,0)$ on the right hand side lie in different fibers of $p$.
	\end{proof}

	We now prove Lemmas \ref{lem:deg(L_chi)=0} through \ref{lem:zero_map}.

	\begin{proof}[Proof of Lemma \ref{lem:deg(L_chi)=0}]
		Let $g$ be a meromorphic section of a flat line bundle. Then $\frac{dg}{g}$ is an honest meromorphic 1-form, because the monodromy in the numerator and denominator cancel. Triangulating the surface and integrating $\frac{dg}{g}$ over each triangle, one finds by the argument principle that $g$ has the same number of zeros as poles. Thus the divisor associated to $g$ has degree zero, and hence the line bundle has degree $0$.
	\end{proof}

	\begin{proof}[Proof of Lemma \ref{lem:chi_2_divisor}]
		By construction, $\sqrt{1-x^{-1}}$ is a meromorphic section of $L_{\chi_2}$, and it has divisor $(1,0) - (0,0)$, so \eqref{eqn:Div(L_{chi_2})} holds. Taking preimages under $f$,
		\begin{align} \label{eqn:Div(f*L_chi_2)}
			\Div(f^*L_{\chi_2})
			\sim f^*[(1,0) - (0,0)]
			= (1,0) + (-1,0) - (0,i) - (0,-i).
		\end{align}
		By Lemma \ref{lem:deg_2_divisor_lemma}, we see that $(0,i) + (0,-i) \sim 2(-1,0)$. Inserting this above yields \eqref{eqn:Div(f^*L_{chi_2})}.
	\end{proof}

	\begin{proof}[Proof of Lemma \ref{lem:injective}]
		Let $\chi$ be in the kernel of the given map $T_E \to \Pic(\Sigma)$, so $f^*L_{\chi}$ is trivial. Then $f_*f^*L_{\chi} = L_{\chi}^{\otimes 2} = L_{\chi^2}$ is trivial. This means that $L_{\chi^2}$ has a nonvanishing holomorphic section. But any holomorphic section is constant by the maximum principle (note the absolute value of a section of $L_{\chi^2}$ is well-defined because $\chi^2$ is a unitary character). Constants have no monodromy, so it follows that $\chi^2 = 1$. Therefore either $\chi = 1$ and we are done, or $\chi = \chi_2$ is the unique nonzero 2-torsion element of $T_E$. But $\chi$ cannot be $\chi_2$, because by Lemma~\ref{lem:deg_2_divisor_lemma} and \eqref{eqn:Div(f*L_chi_2)} in the proof of Lemma~\ref{lem:chi_2_divisor}, $\Div(f^*L_{\chi_2}) \not\sim 0$, contradicting that $\chi$ is in the kernel of $T_E \to \Pic(\Sigma)$.
	\end{proof}

	\begin{proof}[Proof of Lemma \ref{lem:zero_map}]
		We will show that the image of $f_*'f^*$ on the full class group $\Cl(E)$ is $2\mathbb{Z}(1,0) \subseteq \Cl(E')$. This implies the lemma because the degree zero part of $2\mathbb{Z}(1,0)$ is zero.
		
		Let $(a,b) \in E$, viewed as a divisor. Then
		\begin{align*}
			f'_* f^*(a,b)
			= f'_*\Big[\Big(\sqrt{a}, \frac{b}{\sqrt{a}}\Big) + \Big(-\sqrt{a}, -\frac{b}{\sqrt{a}}\Big)\Big]
			= \Big(a, \frac{b}{\sqrt{a}}\Big) + \Big(a,-\frac{b}{\sqrt{a}}\Big)
			\sim 2(1,0),
		\end{align*}
		where the last $\sim$ is by Remark \ref{rem:any_genus}.
		By definition, $\Cl(E)$ is generated by points in $E$, so the image of $\Cl(E)$ under $f'_* f^*$ is generated by $2(1,0)$. Thus the image is $2\mathbb{Z}(1,0)$, as desired.
	\end{proof}

\section{An explicit uniformization}\label{sec:uniformization}
	
	Let $\Sigma$ be the genus 2 curve $y^2 = x^6-1$. In this section we show that $\Sigma$ is arithmetic:
	
	\begin{prop}
		Let $q$ be the indefinite quadratic form $x_1^2 + x_2^2 - 6x_3^2$. Then $\Sigma$ is biholomorphic to $\Gamma \backslash \mathbb{H}^2$, where $\Gamma < \PSL_2(\R)$ is commensurable with $\SO_q(\Z)^+$ inside $\SO_q(\R)^+ \simeq \SO(2,1)^+ \simeq \PSL_2(\R)$.
	\end{prop}

	\begin{proof}
		Consider the map $\Sigma \to \P^1 = \C \cup \{\infty\}$ given by $(x,y) \mapsto x^6$. This is a Galois branched cover, in the sense that $\Aut(\Sigma/\P^1)$ acts transitively on the fibers. In particular, the ramification index is constant on each fiber. There is ramification only over $0,1,\infty \in \P^1$. The ramification index over $1 \in \P^1$ is $2$, and over $0,\infty \in \P^1$ is $6$.
		
		Conformally map the upper half-plane to a hyperbolic triangle with angles $\frac{\pi}{2},\frac{\pi}{6},\frac{\pi}{6}$ such that $1,0,\infty$ map to the vertices with angles $\frac{\pi}{2},\frac{\pi}{6},\frac{\pi}{6}$, in that order. Give the upper half-plane the induced hyperbolic metric. By the Schwarz reflection principle, this metric extends to the lower half-plane, giving a hyperbolic cone metric on $\P^1$ with cone angles $\frac{2\pi}{2}, \frac{2\pi}{6}, \frac{2\pi}{6}$ at $1,0,\infty$. By our computation of ramification indices, this pulls back to a smooth hyperbolic metric on $\Sigma$ (which is compatible with the complex structure). Since $\Sigma \to \P^1$ has degree $12$, we conclude that $\Sigma$ is tiled by $24$ hyperbolic $(2,6,6)$-triangles, and hence that $\Sigma \simeq \Gamma \backslash \mathbb{H}^2$ with $\Gamma$ an index $24$ subgroup of the triangle group $\Delta(2,6,6)$.
		
		Now consider the quadratic space $V$ over $\mathbb{R}$ with basis vectors $e_1,e_2,e_3$ and bilinear form given by
		\begin{align*}
			&\langle e_1,e_1 \rangle
			= \langle e_2,e_2 \rangle
			= \langle e_3,e_3 \rangle
			= 1,
			\qquad
			\langle e_1,e_2 \rangle
			= -\cos\frac{\pi}{2} = 0,
			\qquad\\
			&\langle e_1,e_3 \rangle
			= \langle e_2,e_3 \rangle
			= -\cos\frac{\pi}{6}
			= -\frac{\sqrt{3}}{2}.
		\end{align*}
        {
        The corresponding Gram matrix is
        \begin{align*}
            (\langle e_i,e_j \rangle)
            = \begin{pmatrix}
                1 & 0 & -\frac{\sqrt{3}}{2} \\
                0 & 1 & -\frac{\sqrt{3}}{2} \\
                -\frac{\sqrt{3}}{2} & -\frac{\sqrt{3}}{2} & 1
            \end{pmatrix}.
        \end{align*}
        }
		This has signature $(2,1)$, so the hyperbolic plane can be identified with the space $\mathbb{P}(V)^-$ of negative definite lines in $V$ (this identification is an isomorphism of homogeneous spaces for $\SO_V(\mathbb{R})^+ \simeq \PSL_2(\R)$). Then the projective image of
		\begin{align*}
			\{v \in V : \langle v,e_i \rangle \text{ has the same sign for all } i\}
		\end{align*}
		lands inside $\mathbb{P}(V)^-$, and is a $(2,6,6)$-triangle. The group $\Delta(2,6,6)$ is the subgroup of $\O_V(\mathbb{R})$ generated by reflections across $e_i^{\perp}$; recall that for $e$ a unit vector, the reflection across $e^{\perp}$ is the isometry $v \mapsto v - 2\langle v,e \rangle e$.
		
		Set {$e_1' = 6 e_1$, $e_2' = 6 e_2$, and $e_3' = 2\sqrt{3}e_3$} (this renormalization is helpful to cancel {denominators and} factors of $\sqrt{3}$ in certain inner products).
        {
        Then
        \begin{align*}
            (\langle e_i',e_j' \rangle)
            =
            \begin{pmatrix}
                6 & 0 & 0 \\
                0 & 6 & 0 \\
                0 & 0 & 2\sqrt{3}
            \end{pmatrix}
            \begin{pmatrix}
                1 & 0 & -\frac{\sqrt{3}}{2} \\
                0 & 1 & -\frac{\sqrt{3}}{2} \\
                -\frac{\sqrt{3}}{2} & -\frac{\sqrt{3}}{2} & 1
            \end{pmatrix}
            \begin{pmatrix}
                6 & 0 & 0 \\
                0 & 6 & 0 \\
                0 & 0 & 2\sqrt{3}
            \end{pmatrix}
            =
            \begin{pmatrix}
                36 & 0 & -18 \\
                0 & 36 & -18 \\
                -18 & -18 & 12
            \end{pmatrix}.
        \end{align*}
        }
        Let $\Lambda$ be the lattice in $V$ generated by the $e_i'$. We have $\langle e_i',e_j' \rangle \in \mathbb{Z}$ for all $i,j$, so $\Lambda$ has the structure of a quadratic space defined over $\mathbb{Z}$. Moreover, the reflections across the $e_i^{\perp}$ all preserve $\Lambda$, so
		\begin{align*}
			\Delta(2,6,6) \subseteq \O_{\Lambda}(\mathbb{Z})
			\subseteq \O_{\Lambda}(\mathbb{R}) = \O_V(\mathbb{R})
		\end{align*}
		(the last equality is because $\Lambda_{\mathbb{R}} = V$). Since $\Delta(2,6,6)$ is cocompact in $\O_V(\mathbb{R})$, it must have finite index in $\O_{\Lambda}(\mathbb{Z})$, and hence $\Gamma$ has finite index in $\SO_{\Lambda}(\mathbb{Z})^+$.
		
		Diagonalizing
        {
        the Gram matrix $(\langle e_i', e_j' \rangle)$,
        }
        %$\Lambda_{\mathbb{Q}}$,
        we find that $\Lambda_{\mathbb{Q}}$ is isomorphic
        %over $\mathbb{Q}$
        to the quadratic space $\mathbb{Q}^3$ with diagonal form $q(x) = x_1^2 + x_2^2 - 6x_3^2$:
        {
        \begin{align*}
            \begin{pmatrix}
                36 & 0 & -18 \\
                0 & 36 & -18 \\
                -18 & -18 & 12
            \end{pmatrix}
            =
            \begin{pmatrix}
                6 & 0 & -3 \\
                0 & -30 & 3 \\
                0 & 12 & -1
            \end{pmatrix}^T
            \begin{pmatrix}
                1 & 0 & 0 \\
                0 & 1 & 0 \\
                0 & 0 & -6
            \end{pmatrix}
            \begin{pmatrix}
                6 & 0 & -3 \\
                0 & -30 & 3 \\
                0 & 12 & -1
            \end{pmatrix}.
        \end{align*}
        }Thus $\SO_{\Lambda}(\Z)^+$ is commensurable with $\SO_q(\Z)^+$, and hence $\Gamma$ is commensurable with $\SO_q(\Z)^+$, as desired.
	\end{proof}

\appendix

\section{Embedding a disc of radius \texorpdfstring{$0.13$}{zot} into any hyperbolic surface}\label{app}
In this section, we provide a simple proof for the convenience of the reader of the fact that one can embed a hyperbolic disc of radius $0.13$ isometrically into any hyperbolic surface.

\begin{lem}
     On any closed hyperbolic surface, there exists a point whose injectivity radius is at least $0.13$.
 \end{lem}

 \begin{proof}
     We recall that every closed hyperbolic surface has a pair of pants decomposition and that each pair of pants can be cut into two right angled hexagons. The lemma follows by showing that every right angled hexagon $H$ has a point $p\in H$ such that $\dist(p,\partial H) \geq 0.13$.

     To prove this, first note that by the hyperbolic area formula for polygons, the area of any right angled hexagon is $\pi$. Now we cut our hexagon into $4$ triangles, one of which has area at least $\pi/4$. Our claim follows by showing that the inradius of any such hyperbolic triangle $T$ is at least $0.13$. 

     To see this, let us denote our triangle $T$, with vertices $A,B,C$. Let the incenter be denoted by $O$ and let the incircle touch the three sides at $P,Q,R$. We can now partition $T$ into $6$ right angled triangles, one of which has area at least $\pi/24$. Let this triangle have vertices $A,O,P$ and we use corresponding lower-case letters to denote the angles at the respective vertices. Note that $p=\pi/2$ and the inradius is $\dist(O,P)$ which is minimized when $a=0$. When $a=0$, and so $o \leq \pi/2 - \pi/24$, we have that
     $$ \sin(\pi/24) \leq \cos(o) = \frac{\tanh{(\dist(O,P))}}{\tanh{(\dist(O,A))}} = \tanh{(\dist(O,P))}\,.$$
     So, we get that $\dist(O,P) \geq \tanh^{-1}(\sin(\pi/24)) > 0.13$.
 \end{proof}

\printbibliography

\end{document}